\documentclass{amsart}  
\RequirePackage[OT1]{fontenc}
\RequirePackage{verbatim, amscd,epsfig,amsmath,amsthm,amstext,amssymb,tikz}
\usepackage{url}
\usepackage{mathrsfs}
\usepackage[all]{xy}
\usepackage{lineno}
\usepackage{tikz}
\usetikzlibrary{patterns}
\usepackage{caption}
\newcommand{\R}{\mathbb{R}}
\usepackage{graphicx}
\usepackage{subcaption}
\usepackage{mwe}

\theoremstyle{plain}
\newtheorem{thm}{Theorem}[section]
\newtheorem{lemma}[thm]{Lemma}
\newtheorem{prop}[thm]{Proposition}

\newtheorem{defn}[thm]{Definition}
\newtheorem{example}[thm]{Example}
\newtheorem{rmk-defn}[thm]{Remark}
\newtheorem{rmk-thm}[thm]{Remark}
\newtheorem{rmk-lem}[thm]{Remark}
\newtheorem{rmk-prop}[thm]{Remark}

\title[Digital tools for non-diffeomorphic shapes]{Digital Euler Characteristic Transform}

\author{Henry Kirveslahti}
\address{Laboratory for Topology and Neuroscience, EPFL, Lausanne, Switzerland;  IMADA \& Center for Quantum Mathematics, University of Southern Denmark, Odense, Denmark}
\email{hklahti@imada.sdu.dk}

\author{Xiaohan Wang}
\address{Laboratory for Topology and Neuroscience, EPFL, Lausanne, Switzerland; School of Physical and Mathematical Sciences, Nanyang Technological University, Singapore}
\email{xiaohan005@e.ntu.edu.sg}


\begin{document}

\begin{abstract}
The Euler Characteristic Transform (ECT) of Turner et al. provides a way to statistically analyze non-diffeomorphic shapes without relying on landmarks. In applications, this transform is typically approximated by a discrete set of directions and heights, which results in potential loss of information as well as problems in inverting the transform. In this work we present a fully digital algorithm for computing the ECT exactly, up to computer precision; we introduce the \textsc{Ectoplasm} package that implements this algorithm, and we demonstrate this is fast and convenient enough to compute distances in real life data sets. We also discuss the implications of this algorithm to related problems in shape analysis, such as shape inversion and sub-shape selection. We also show a proof-of-concept application for solving the shape alignment problem with gradient descent and adaptive grid search, which are two powerful methods neither of which is possible using the discretized transform.
\end{abstract}

\maketitle

\section{Introduction}
\label{sec:intro}

Shapes are difficult to analyze statistically, because they have no obvious metric structure needed for statistical machinery. Without a metric structure, we can not define such basic statistical quantities like mean of two shapes. A standard idea in shape analysis is to pick points on the shapes, declare correspondence between them, and define the distance between shapes to be the distance between the matrices defined by the landmarks, modulo appropriate Euclidean actions. This leads to what is called the \emph{Kendall Shape Space}. Using landmarks leads to loss of information, because shapes are not point clouds, and there is no canonical choice for picking the correspondence between the landmarks. This is especially relevant for the modern age, where our shape data are represented in high fidelity in online repositories.

Another idea is that of diffeomorphism-based shape alignment. This corresponds to a digitalization of Kendall Shape Space as in this approach the shapes are represented as surfaces and their distances are studied via correspondence maps, continuous objects that can achieve arbitrary precision. A key restriction for there to be such a correspondence map is that the shapes must be diffeomorphic. This assumption is not realistic for a multitude of shape data, such as skulls that may have different number of cavities or medical data such as brain MRI.

A third idea for shape analysis is based on integral geometry and geometric measure theory. A powerful tool in this is the Euler Characteristic Transform (ECT), which is mathematically a topological Radon transform. This transform is the subject of this work. This framework allows for statistical analysis beyond landmarks without relying on the diffeomorphism assumption.

All existing methods for computing the ECT rely on evaluating the ECT on a discrete grid of directions, and most of the time, also the heights. This is somewhat undesirable if our goal is to define a digital tool for shape analysis - such an approach still requires us to pick a discrete set of directions, and our inference is based on these, a setup much similar as with landmarks. While there is a theoretical notion of a Moduli space for shapes that can be told apart from each other by finitely many directions \cite{CMT}, this number is very large, indeed much larger than what is used in practice. Also, the result states that such transform is injective for these shapes, but there is to date even a theoretical guarantee on how many directions give a good notion of distance for a given collection of shapes. In practice, this can be only solved by picking more and more directions until the distances seem to stabilize.

In this work we introduce the \textsc{Ectoplasm}, a fully-digital shape analysis tool for dissimilar shapes. The work relies heavily on the observations made in \cite{CMT} and folklore results regarding the analytic form of the transform. We push these results further to develop a fully digital ECT algorithm for shape analysis. In some regard, this work complements the theoretical approach to unify shape theory via sheafification presented in \cite{arya2024sheaf}, with the distinction that from a theoretical perspective the Persistent Homology Transform (PHT) has nicer properties than the ECT, which is preferable from applied perspective and subject of this work.

The main advantages of this framework are threefold: First, the downstream analysis, for example statistical inference, is improved, because the digital transform makes use of all the data. This is especially relevant for getting accurate information on which shapes are closest to each other, a key ingredient in manifold learning applications that rely on small distances to approximate the tangent space. Second, not having to discretize the transform obviates the problem of having to choose the discretization parameters, a major headache for the practitioners. Third, as the transform is computed in its exact form, the rich mathematical structure of the transform can be exploited to its full extent. For example, the digital algorithm allows for a genuine $O(d)$-invariant action. We will discuss these discretization related issues and ways to circumvent them in Section \ref{sec:background}. Finally, the digital approach also gives us a ground truth, which can be used to assess how accurate a given discretization is.

The paper is organized as follows: in Section \ref{sec:background} we review some of the background on ECT and its typical use cases. In Section \ref{sec:algorithms} we give a simplified, brute-force version of the digitalization algorithm in two dimensions.This serves as an introduction to Section \ref{sec:3d}, where we present the 3D-algorithm. The technical details pertaining to spherical integrals are in \ref{secA1}. In Section \ref{sec:case study} we apply the the algorithm to perform exploratory data analysis on a real life data set, and compare the results against the conventional, discretization based method, as well as previously published methods based on landmarks and diffeomorphisms. In Section \ref{sec:alignment}, we study the two aforementioned solutions to the alignment problem that the digital transformation allows for.

\section{Background on the ECT}
\label{sec:background}
In this section we provide a brief background on the ECT and its use cases in applications. We will take as concrete approach as possible when introducing the background material, a more mathematically oriented reader can consult, for example \cite{CMT} for a more general treatment of the subject. All our objects are finite and compactly supported.

The applied upshot of the ECT is that it allows us to compare shapes. This statement necessitates the question of what constitutes a shape. In general, the ECT framework is defined in the flexible language of \emph{$o$-minimal} structures, but as this work is mainly concerned with, and indeed only applicable to, piecewise linear meshes, we will state the relevant results in this language. This is also the relevant framework for practical considerations, because computer standard architecture prefers triangles over curved geometric objects.

Our meshes will be finite geometric simplicial complexes, with two meshes declared identical if they are identical as subsets of $\R^d$. We will recall the somewhat technical definition of a geometric simplicial complex.

\begin{defn}
A (finite) abstract simplicial complex $\Sigma$ is a union of subsets of $\big[N \big]$ (called simplices) that is closed under taking subsets, i.e. if $\tau \le  \sigma$ and $\sigma \in \Sigma$, then $\tau \in \Sigma$. The singletons of $\Sigma$ are called its \emph{vertices}, and the dimension of a simplex is its cardinality minus one.
\end{defn}

\begin{defn}
Let $\Sigma$ be a finite abstract simplicial complex with vertices $x_1,\ldots,x_n$ and $M$ simplices. A \emph{geometric realization map} of $\Sigma$ is an affine map $f$ from a subset of $\R^n$ to $\R^d$ (for $d$ some positive integer), whose domain consists of

\begin{enumerate}
\item The coordinate vectors $e_i$;
\item Any strictly convex combination of a collection of coordinate vectors \\ $\{e_{j_1},e_{j_2},\ldots,e_{j_m}\}$, whenever $\{x_{j_1},x_{j_2},\ldots,x_{j_m} \}$ is a simplex of $\Sigma$.
\end{enumerate}
\end{defn}

Note that by affinity, the geometric realization is completely determined by the image of the coordinate vectors $e_i$, and the combinatorial structure of $\Sigma$.

\begin{defn}
A \emph{geometric simplicial complex} $X$ is an abstract simplicial complex $\Sigma$, together with an injective geometric realization map $f$.
\end{defn}

For our purposes, this definition is too fine, because we are only interested in meshes as subsets of the Euclidean space, typically $\R^3$. Hence we don't care about the combinatorial structure beyond that. This conundrum was also discussed in \cite{CMT}, Remark 5.12.

\begin{defn}
A \emph{mesh} is a geometric simplicial complex as a subset of $\R^d$. 
\end{defn}

We will use the word mesh to distinguish it from a shape, which is usually regarded as a mesh, modulo certain Euclidean symmetries, such as rotations, translations and scalings.

We will require our meshes to be supported in the unit ball. This restriction, albeit unsatisfactory from the mathematical perspective, is well-grounded in the context of Kendall shape space and the requirement that scalar multiples of shapes are equal, as well as the practical wisdom that a collection of meshes should be scaled to the same measure before they are analyzed.

One of the important invariants of a simplicial complex $X$ is its Euler Characteristic. It is indeed a topological invariant of the complex, but for our purposes we can take it to be the alternating sum of its simplex counts.

\begin{defn}
The Euler Characteristic $\chi(X)$ of a geometric simplicial complex $X$ with $b_0$ $0$-simplices (vertices), $b_1$ 1-simplices (edges), and in general $b_j$ $j$-simplices, $j=0,\ldots, n$, is
$$
\chi(X) = \sum_{j=0}^{n} (-1)^{j} b_j.
$$
\end{defn}

The following fact is a consequence of the $o$-minimal Triangulation Theorem from \cite{tametopology}.

\begin{prop}
If mesh $M$ is obtained from two different geometric simplicial complexes $X$ and $Y$, then $\chi(X) =\chi(Y)$, so that $\chi(M)$ is well-defined.
\end{prop}

\begin{defn}
\label{def:filtered}
For a mesh $X \subset \R^d$, and a direction $v \in S^{d-1}$, we can subset $X$ at height $h$:

$$
X_{v}(h) = \{ x \in X | x \cdot v \le h \}.
$$
\end{defn}
We will call the subset $X_{v}(h)$ the \emph{ filtration of $X$ to height $h$ at direction $v$}. 

For a fixed $v$, this is actually a \emph{filtration} of topological spaces indexed by $h \in \R$ in that $X_v(h_1) \subseteq X_v(h_2)$ whenever $h_1 \le h_2$. 

Note that the restriction $\{ x \in \R^d | x \cdot v \le h, x \in X \}$, is a piecewise linear set, but it need not be a subcomplex of $X$. However, we can effectively treat it as such, which is the subject of the following definition and theorem.

\begin{defn}
\label{def:restriction}
Write $X_{v,h}$ for the subcomplex of $X$ consisting of the simplices that are contained in the half-space
$$
H_{v,h} = \{ x \in \R^d | x \cdot v \le h \}.
$$

We call $X_{v,h}$ the \emph{ restriction of $X$ to height $h$ at direction $v$}. 
\end{defn}

We are now ready to define the Euler Characteristic Transform, the main subject of this work:

\begin{defn}
\label{ect}
Let $X \subset \R^d$ be a mesh, and $X_{v,h}$ the subcomplex from Definition \ref{def:restriction}. The  \emph{Euler Characteristic Transform} $\textrm{ECT}_X$ of $X$ is an integer-valued function on $S^{d-1} \times \R$ defined as

$$
ECT_X(v,h) = \chi(X_{v,h}).
$$
\end{defn}

While we don't need the general definition for Euler Characteristic of a topological space for computing anything, we note the following result which is a consequence of \cite{bestvina1997morse}, and it justifies the filtration viewpoint of Definition \ref{ect}. We state this result for reference: 

\begin{thm}
For a mesh $X \subset \R^d$, and a direction $v \in S^{d-1}$, and height $h \in \R$, we have:
$$\chi(X_{v}(h)) = \chi(X_{v,h}).$$
\end{thm}

The upshot of a more theoretical framework is that by a well-known consequence of the work by Schapira \cite{Schapira:tom} is that $ECT_X(v,h)$ determines $X$ completely. This is also relevant for what is done in this paper, as the results that we need follow from this framework. However, we do not need this machinery for the exposition of the work that we do here, readers are referred to \cite{CMT} for the details.

\subsection{The ECT distance}

With the ECT, we can define the ECT inner product between two meshes $X$ and $Y$ as

\begin{equation}
\langle X,Y \rangle  = \Big( \int_{-1}^{1} \int_{S^{d-1}} \textrm{ECT}_X(v,h)\textrm{ECT}_Y(v,h) \ dv dh \Big).
\label{eq:inner_product}
\end{equation}

This of course gives us a notion of norm, or distance between two meshes:
\begin{equation}
||X-Y||^2  = \langle X-Y,X-Y \rangle  =  \langle X,X \rangle - 2 \langle X,Y \rangle + \langle Y,Y \rangle.
\label{eq:norm}
\end{equation}

The expression (\ref{eq:norm}) of squared distance above is especially convenient for our purposes. We note of the equivalent expression, which is more common in the literature:

\begin{equation}
d_{ECT}(X,Y)  = \Big( \int_{-1}^{1} \int_{S^{d-1}} \big( \textrm{ECT}_X(v,h)-\textrm{ECT}_Y(v,h) \big)^2 \ dv dh \Big)^{1/2}.
\label{eq:dist}
\end{equation}

Computing the distances in  Equation (\ref{eq:dist}) seems like a daunting task, for the integral is taken over the sphere. In practice, this expression is always approximated via some sort of discretization. While the integral indeed is very complicated, with the help of a computer it can be computed analytically. An effective algorithm for doing so is the main result of this paper. Before introducing this algorithm, we will make a few comments on the discretization and its properties.

\subsection{Discretization of ECT}

A straightforward way to discretize the transform and to make Equation (\ref{eq:dist}) computable is to pick a set of direction $v_1,v_2,\ldots,v_n$ and a set of heights $h_1,\ldots,h_m$ and then represent the (discretized) transform as an $n \times m$ vector (or matrix). For these discretized objects, (\ref{eq:dist}) is nothing more than the $L_2$ norm of the matrices. If we hook this idea back to landmarks, we are ultimately performing a similar analysis: Our inference is based on matrices, that represent finitely many directions evaluated at finitely many heights. This bears resemblance to the landmark problem - how do we choose the directions? We will discuss the merits and demerits of the approximation in this subsection.

One of the main results in this direction is the complexity class result in \cite{CMT}, Theorem 7.14. -- that there is an upper bound on how many directions are required to tell two shapes in a certain complexity class apart. This is a mathematical result, meaning that by picking that many directions guarantees that the discretized transform is injective on the complexity class. This result, being one of the only ones in the literature, is important, but we don't have sufficient understanding on what is the best procedure to pick the discretization in practice for a collection of shapes. We can not, for example, make statements that $N$ many directions explain $99\%$ of the variability in distance, or to say which directions should be chosen for best representation of the shapes if we can only afford $n$ many directions. The lack of these type of results makes using these methods harder, because there is no clear cut way to choose the discretization parameters.

This problem is exacerbated when the statistical question is something more complicated than regression or classification. If our only goal is to do classification, and we use a kernel method, we only need to worry about approximating the distance matrix well, because all of our inference is based on that. Then the conventional wisdom to pick $n$ and $m$ as large as possible, and pick the directions (and heights) as uniformly as possible, seems to be a reasonable advice. On the other hand, if we are interested in performing a more complicated task, such as feature selection, which was done in \cite{SINATRA}, then the choice of discretization parameters becomes a double-edged sword. One one hand, we would like to increase $n$ and $m$ to make the distances more accurate. But on the other hand, if we pick $n$ and $m$ too large, the variation in the $n \times m$ regression parameters can too homogeneous, causing at best multicollinearity, and at worst, making the parameters not even estimable. 

The digitalization procedure proposed in this work alleviates these two problems: By not using a discretization, the distances are computed (up to computer precision) in closed form, so we know our inference is as good as it can get. Also, the digitalization obviates the need to choose the discretization parameters in the first place, so that multicollinearity of features will not be an issue. Of course, one needs to rework all the statistical machinery for variable selection, because with the digital transform we replace vectors with functions, so that the statistical machinery must involve functional data analysis. This interesting problem is left for future research.

Another advantage that the digital version has is that it is better suited for capturing small differences in the meshes. This is especially relevant for applications that make use of tangent space approximations, such as ISOMAP. Because smaller distances are used for constructing the tangent space, a method that is more accurate in getting the small distances right will give more accurate inference.

Furthermore, the discretized version suffers a small inconvenience when it comes equivariance. Mathematically, the ECT is an $O(d)$-equivariant transform, meaning that for a rotation matrix (possibly including reflection) $A$, we have

$$
\textrm{ECT}_{AX}(v,h) = \textrm{ECT}_{X}(Av,h).
$$

If one uses the discretized version, one only has access to finitely many directions. In a typical use case, these directions do not form an orbit under $A$, so that we don't have access $\textrm{ECT}_{X}(Av,h)$ from $\textrm{ECT}_{X}(v,h)$. This means they must be somehow approximated, and this complicates the analysis and can lead to costly computations. See \cite{SINATRA} Section 4 in the supplementary material for details.

On the other hand, from the digitalized version, $\textrm{ECT}_{X}(Av,h)$ is readily available from $\textrm{ECT}_{X}(v,h)$. What is more, the alignment penalty map

$$
A \mapsto -2 \langle X,AY \rangle
$$

is, for generic meshes $X$ and $Y$, almost everywhere differentiable. This phenomenon has been discussed in e.g. \cite{leygonie2022framework}. This means that in theory, with the digital transform, the alignment problem is amenable to a gradient descent type solution (rather than grid search). We study this further in Section \ref{sec:alignment}.

Finally, the discretized version has no hope of using the Schapira's inversion formula to go from transforms back to shapes. With the digitalized version of the transform, it is conceivable that Schapira-type inversion formula could be used for pulling back evidence back on the original problem.

\section{The Algorithm}
\label{sec:algorithms}

The purpose of this Section is to explain how we can evaluate Expression (\ref{eq:dist}) analytically. To do so, we first introduce a brute force algorithm for the 2D case, where the formulae are simpler and the transform is easier to visualize. We postpone the actual 3D-algorithm with additional refinements to Section \ref{sec:3d}.
The algorithm, while quite complicated, makes use of several well-known facts of the ECT. We will state these facts shortly. With these facts, we will describe the transforms (or equivalently the meshes) in what we will call the \emph{proto-transfom} format. This format is a bridge between the mesh and its transform, and a key to the digitalization algorithm.

The first well-known fact about the ECT that we want to state is that for a piecewise linear mesh, the Euler Curve $f_v(h)$ can only change at the vertices.

\begin{prop}
\label{propa1}
Let $X$ be a piecewise linear mesh with vertices $\{x_1,\ldots, x_n\}$ and $v$ a direction such that $\langle x_i,v \rangle  \neq \langle x_j,v \rangle $ for $i \neq j$ and $f_v(h)$ the Euler curve in direction $v$. Then $f_v(h)$ can only change at the vertices, that is at heights $h_i$ given by $h_i=\langle v,x_i \rangle $ for some vertex $x_i$ of $X$. 
\end{prop}

\begin{defn}
\label{defa1}
Following the notation and assumptions of Proposition \ref{propa1} , vertex $x_i$ is called \emph{Euler}-critical (with respect to $v$) and $h_i$ is called a \emph{critical height}. We will call

$$
a_i(v) := f_v(h_i)-f_v(h_i-\epsilon),
$$
where $\epsilon$ is a small positive number, for example $\textrm{min}_{j,k} \langle x_j - x_k,v \rangle/2$,

the \emph{gain} associated to vertex $x_i$ in direction $v$.
\end{defn}

Next we describe a consequence of the stratification result in \cite{CMT}. The argument stated here has also appeared previously in \cite{PHT}.

\begin{prop}
\label{propa2}
Let $X$ be a mesh with vertices $\{x_1,\ldots, x_n\}$. If two directions $v$ and $\nu$ satisfy

$$
\langle v,x_i \rangle  \le \langle v,x_j \rangle  \textrm{ if and only if } \langle \nu,x_i \rangle  \le \langle \nu,x_j \rangle
$$

for all $i,j \in [n]$, the directions $v$ and $\nu$ are equivalent, and they have the same critical vertices with same gains.
\end{prop}

\begin{proof}[Proofs of Propositions \ref{propa1}, \ref{propa2}]
See Propositions 5.18 and Lemma 5.19 of \cite{CMT}.
\end{proof}

Next we will state a simple fact of the height function with an explicit formula in dimension 2.

\begin{prop}
\label{2dheight}
Let $p=(x,y)$ be a vertex in $\R^2$. Its height function $h$ with respect to $v$ parametrized as $v=\big(\cos(t),\sin(t) \big)$ is given by
$$
h_p(v) = p \cdot v   = x \cos(t) + y  \sin(t),
$$

so that its integral over a polygon $P_k=[\tau,\theta]$ is given by:

\begin{align*}
I_{P_k}(p) & = \int_{\tau}^{\theta} \cos(t) x + \sin(t) y \ dt \\
&= \Big( \cos(\tau) - \cos(\theta) \Big)y + \Big( \sin(\theta) -\sin(\tau) \Big)x.
\end{align*}
\end{prop}
\begin{proof}
A straightforward calculation.
\end{proof}

With these facts we can now define the proto-transform format, as well as a brute force algorithm for computing it in two dimensions. The idea behind the proto-transform format is to list the finite pieces of the transform in a systematic format that allows us to compute Expression (\ref{eq:dist}) with ease.

Let $X$ be a 2D-mesh with vertices $\{x_1,\ldots, x_n\}$.
\begin{itemize}
\item Solve $\langle v,x_i \rangle  = \langle v,x_j \rangle $ for all $i \neq j$, get $v_{i,j}, -v_{i,j}$
\item Order the $v_{i,j}^{\cdot}$ by their angle, get pieces $P_k=[\theta_k,\theta_{k+1}]$, $k=1,\ldots,N$.
\item Set $\alpha_k= \frac{\theta_k+\theta_{k+1}}{2}$
\item Look at the order of $x_i$ under $\langle \alpha_k, x_i \rangle $.
\item For each $x_i$, evaluate the discrete ECT at height $h_i + \epsilon$, i.e. just above $x_i$ (with e.g. $\epsilon$ from Definition \ref{defa1}). Record this value to $a_{k,i}$, and record $x_i$ to $b_{k,i}$.
\item The ECT of $X$ is now represented as $\bigcup_{k=1}^{N} (P_k, \textbf{a}_k, \textbf{b}_k )$.
\end{itemize}

We call the the collection $\bigcup_{k=1}^{N} (P_k, \textbf{a}_k, \textbf{b}_k )$ the \emph{proto-transform} format of $X$. Next we will show how we can use this format to compute integrals of the form \ref{eq:dist} in 2-dimensions.

The first observation is that inside the pieces $P_k$, the order of the vertices stays the same. Then, inside each  $P_k$ (and heights $[-1,1]$), the $ECT_X(v,h)$ is a simple function described as

$$
ECT_{X | P_k}(v,h) = \sum_{i=1}^{n_k} a_{i,k} \big( 1_{ \ge \langle v,b_{k,i} \rangle }(h)1_{ \le \langle v,b_{k,i+1} \rangle }(h) \big) ,
$$

where we define $\langle v,b_{k,n_k+1} \rangle$ to be identically 1.

This expression is again nothing new, it is equivalent to the formula in Lemma 5.19 of \cite{CMT}, expressed in explicit coordinates.

This is a simple function of the height functions from Proposition \ref{2dheight}, whose integral is readily available.

To compute the (squared) distance in Equation (\ref{eq:norm}), we first note that each of the three pieces

\begin{equation*}
d_{ECT}(X,Y)^2 = ||X-Y||^2 = \langle X,X \rangle - 2 \langle X,Y \rangle + \langle Y,Y \rangle.
\end{equation*}

are just simple functions of the height functions from Proposition \ref{2dheight}, because the product of simple functions is simple.

Explicitly, if the ECT of $X$ is represented as $\bigcup_{k=1}^{N} (P_k, \textbf{a}_k, \textbf{b}_k )$ and we write $P_k = [\theta_k,\theta_{k+1}]$, the first term $\langle X,X \rangle$ can be expressed as: 

\begin{align*}
\langle X,X \rangle & = \int_{-1}^{1} \int_{S^{d-1}} \sum_{k=1}^{M} \sum_{i=1}^{n_k} a_{i,k}^2 \big( 1_{ \ge \langle v,b_{k,i} \rangle }(h)1_{ \le \langle v,b_{k,i+1} \rangle }(h) \big) \ dv \ dh \\
& = \sum_{k=1}^{M} \int_{-1}^{1} \int_{P_k}  \sum_{i=1}^{n_k} a_{i,k}^2 \big( 1_{ \ge \langle v,b_{k,i} \rangle }(h)1_{ \le \langle v,b_{k,i+1} \rangle }(h) \big)  \big) \ dv \ dh \\
& = \sum_{k=1}^{M} \Big( a_{n_k,k}^2 (\theta_{k+1}-\theta_k ) - a_{1,k}^2 I_{P_k}(b_{1,k}) + \sum_{i=2}^{n_k}I_{P_k}(b_{i,k}) \big(a_{i+1,k}^2-a_{i,k}^2 \big) \Big).
\end{align*}

Here $I_{P_k}(b_{i,k}) $ is the integral over the height function of the vertex $b_{i,k}$ over polygon $P_k$, which was explicitly computed in Proposition \ref{2dheight}. The integral is then just a sum of trigonometric functions, and analytically tractable.

Similar reasoning applies for the other two terms, but note that the cross-product term in the middle requires us to overlay the pieces $P_k$ of $X$ and $Q_k$ $Y$ to a common subdivision. See Example \ref{example1} for an explicit computation.

Note that the algorithm that we have described in this section is just for illustration purposes. In fact, it is very inefficient. For two meshes with $n$ and $m$ vertices, we would first need to compute $\binom{n}{2}$ and $\binom{m}{2}$ pairwise intersections to get $P_k, Q_k$, and then we still have to perform the expensive operation of overlaying these two triangulations. Luckily, this complexity can be greatly reduced. As a matter of fact, we don't have to compute all the pairwise intersections, and we can also make the pieces $P_k$ larger than what was described here. We will discuss these improvements in the next section, where we also introduce the algorithm for 3 dimensions, which is also more interesting from the perspective of real-life applications.

\begin{example}
\label{example1}
\begin{figure}[!ht]
\includegraphics[width=\linewidth]{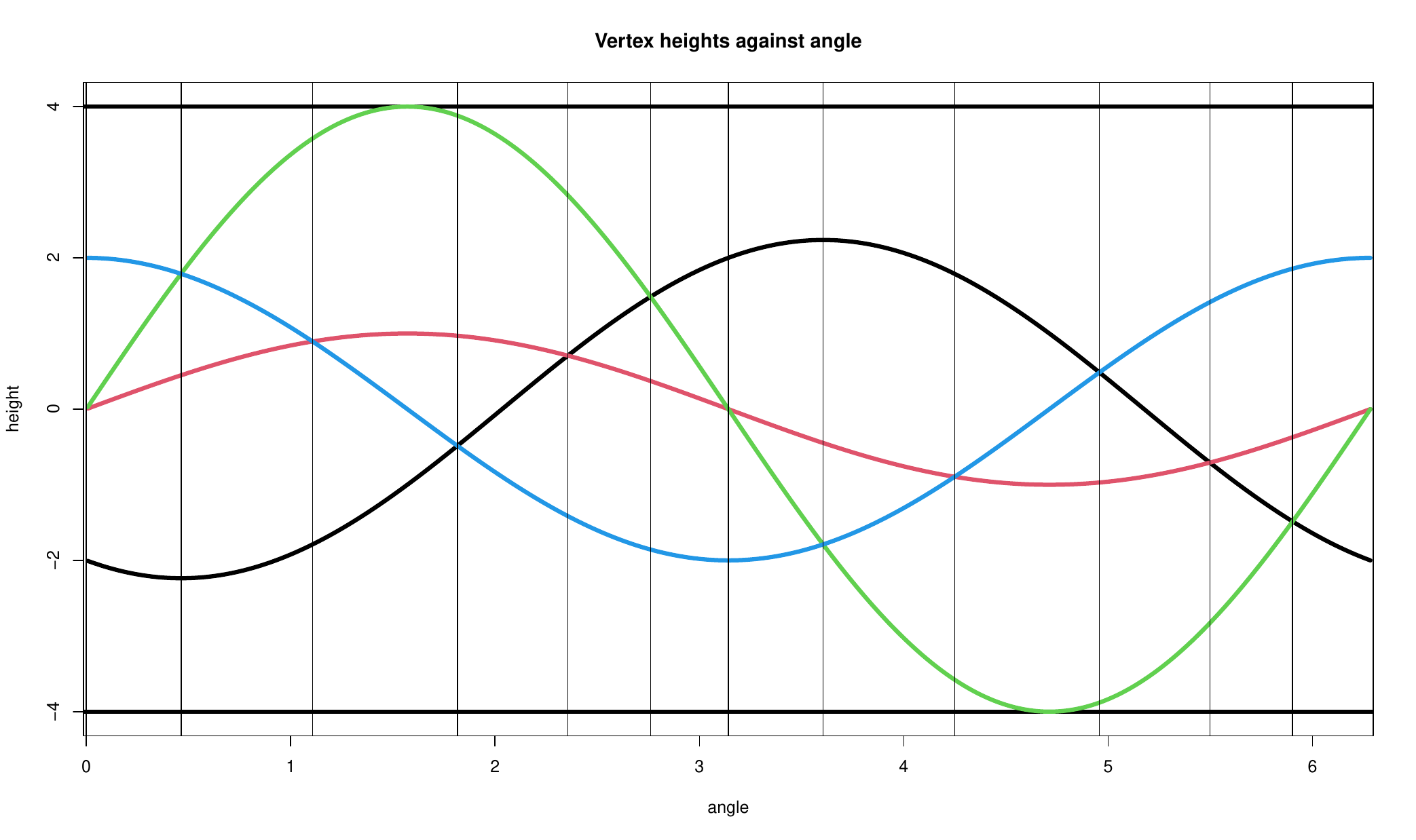}  
\caption{Example \ref{example1}, the height functions of the 4 vertices. The vertical lines represent equi-height angles, or endpoints of each $P_k$. Vertices 1,2,3 and 4 are represented in black, red, green and blue, respectively.}
	\label{ex1a}
\end{figure}

In this example we compute the ECT inner product of a 2D-triangle by hand, following the brute force algorithm. Let $\Sigma$ be an abstract simplicial complex $\{(123,234)\}$, and consider its geometric realization $X$  given by map $f_X: e_1 \mapsto (-2,-1), e_2 \mapsto (0,1), e_3 \mapsto (0,4)$ and $e_4 \mapsto (2,0)$. For simplicity of computations, we will not scale $X$ to the unit ball (nor center it), instead we shall carry the integration out over the origin-centered ball of radius $4$.

We first proceed by solving the equi-height angles. For two pairs of vertices $(x_1,y_1)$ and $(x_2,y_2)$, these are given by $\theta = \textrm{atan}(D) = \textrm{atan} \big(\frac{x_1-x_2}{y_2-y_1}\big)$ These are given in the Table \ref{table:1} below, and also depicted in Figure \ref{ex1a}.

\begin{table}[!ht]
\caption{Auxiliary computations for solving equi-height angles}
\label{table:1}
\begin{center}
\begin{tabular}{ |c|c|c|c|c|c| } 
 \hline
 $x_1$ & $y_1$ & $x_2$ & $y_2$ & $D$ & Name \\ \hline
 $-2$ & $-1$ & $0$ & $1$ & $-1$ & $D_{12}$ \\ \hline
 $-2$ & $-1$ & $0$ & $4$ & $-\frac{2}{5}$ & $D_{13}$ \\ \hline
 $-2$ & $-1$ & $2$ & $0$ & $-4$ & $D_{14}$ \\ \hline
 $0$ & $1$ & $0$ & $4$ & $0$ & $D_{23}$ \\ \hline
 $0$ & $1$ & $2$ & $0$ & $2$ & $D_{24}$ \\ \hline
 $0$ & $4$ & $2$ & $0$ & $\frac{1}{2}$ & $D_{34}$  \\
 \hline
\end{tabular}
\end{center}

\end{table}

Ordering the angles from table \ref{table:1}, and investigating the order $\textbf{b}_k$ of the vertices $x_i$ inside each $P_k$ together with the Euler characteristic $\textbf{a}_k$, we can represent the ECT of $X$ in a table (Table \ref{table:2}), with $P_k=[\theta_k,\theta_{k+1}]$:

\begin{table}[!ht]
\caption{The proto-transform format of mesh $X$. Each row lists a polygon, together with the order of vertices inside the polygon ($\textbf{b}_k$), and ($\textbf{a}_k$) lists the Euler characteristic after each vertex. For convenience, the angles are expressed in terms of trigonometric functions.}
\label{table:2}
\begin{center}
\begin{tabular}{ |c|c|c|c|c|c|c| } 
 \hline
$k$ & $ \tan(\theta_k) $ & $\tan(\theta_{k+1})$ & $\cos(\theta_k)$ & $\sin(\theta_k)$ & $\textbf{b}_k$ & $\textbf{a}_k$ \\ \hline
 $1$ & $0$ & $1/2$ & $1$ & $0$ &  $(p_1,p_2,p_3,p_4)$ & $(1,1,1,1)$ \\ \hline
 $2$ & $1/2$ & $2$ & $2/\sqrt{5}$ & $1/\sqrt{5}$ &  $(p_1,p_2,p_4,p_3)$ & $(1,1,1,1)$ \\ \hline
 $3$ & $2$ & $-4$ & $1/\sqrt{5}$ & $2/\sqrt{5}$ &  $(p_1,p_4,p_2,p_3)$ & $(1,2,1,1)$ \\ \hline
 $4$ & $-4$ & $-1$ & $-1/\sqrt{17}$ & $4/\sqrt{17}$ &  $(p_4,p_1,p_2,p_3)$ & $(1,2,1,1)$ \\ \hline
 $5$ & $-1$ & $-2/5$ &$-1/\sqrt{2}$ &$1/\sqrt{2}$ &  $(p_4,p_2,p_1,p_3)$ & $(1,1,1,1)$ \\ \hline
 $6$ & $-2/5$ & $0$ & $-5/\sqrt{29}$ & $2/\sqrt{29}$ &  $(p_4,p_2,p_3,p_1)$ & $(1,1,1,1)$ \\ \hline
 $7$ & $0$ & $1/2$ & $-1$ & $0$ &  $(p_4,p_3,p_2,p_1)$ & $(1,1,1,1)$ \\ \hline
 $8$ & $1/2$ & $2$ & $-2/\sqrt{5}$ & $-1/\sqrt{5}$ &  $(p_3,p_4,p_2,p_1)$ & $(1,1,1,1)$ \\ \hline
 $9$ & $2$ & $-4$ & $-1/\sqrt{5}$ & $-2/\sqrt{5}$  &  $(p_3,p_2,p_4,p_1)$ & $(1,1,1,1)$ \\ \hline
 $10$ & $-4$ & $-1$ & $1/\sqrt{17}$ & $-4/\sqrt{17}$&  $(p_3,p_2,p_1,p_4)$ & $(1,1,1,1)$ \\ \hline
 $11$ & $-1$ & $-2/5$ &$1/\sqrt{2}$ &$-1/\sqrt{2}$ &  $(p_3,p_1,p_2,p_4)$ & $(1,1,1,1)$ \\ \hline
 $12$ & $-2/5$ & $0$ &$5/\sqrt{29}$ & $-2/\sqrt{29}$ &  $(p_1,p_3,p_2,p_4)$ & $(1,1,1,1)$ \\ \hline
\end{tabular}
\end{center}

\end{table}

Table \ref{table:2} contains all the information to compute $\langle X, X \rangle$. This decomposition contains a lot of polygons and the expression could be much simplified. This inefficiency is a consequence of using the brute force algorithm. In fact, we see that the integrand is 1 except for parts between $p_2$ and $p_4$ in polygon $P_3$ and between $p_2$ and $p_1$ in polygon $P_4$.

From Table \ref{table:2}, we can readily compute $\langle X, X \rangle$. The integrand is 1 from the minimum vertex until height $4$, except for the two polygons $P_3$ and $P_4$ where the squared Euler characteristic is $4$ after $p_1$ (in $P_3$) and $p_4$ (in $P_4$) up until $p_2$. In other words:
\begin{footnotesize}
\begin{align*}
\langle X,X \rangle & = \int_{-4}^{4} \int_{S^{d-1}} \sum_{k=1}^{12} \sum_{i=1}^{4} a_{i,k}^2 \big( 1_{ \ge \langle v,b_{k,i} \rangle }(h)1_{ \le \langle v,b_{k,i+1} \rangle }(h) \big) \ dv \ dh \\
& = \sum_{k=1}^{12} \int_{\theta_{k}}^{\theta_{k+1}} \big(4 - \min_i ( p_i \cdot v ) \big) \ dv + 4 \Big[ \int_{\theta_2}^{\theta_3} \big( p_2 \cdot v - p_4 \cdot v \big) \ dv + \int_{\theta_3}^{\theta_4} \big( p_2 \cdot v - p_1 \cdot v \big) \ dv \Big] \\
& = 4 \Big[ 2 \pi + I_{P_3}(p_2) + I_{P_4}(p_2) - I_{P_3}(p_4) - I_{P_4}(p_1) \Big] \\
& - \sum_{k=\{1,2,3,12\}} I_{P_k}(p_1) - \sum_{k=4}^{7} I_{P_k}(p_4) - \sum_{k=8}^{11} I_{P_k}(p_3) \\
& = 8 \pi + 4\big( \frac{1}{\sqrt{5}} + \frac{1}{\sqrt{17}} \big) + 4\big(-\frac{1}{\sqrt{17}} + \frac{1}{\sqrt{2}} \big) -8\big( \frac{4}{\sqrt{17}} - \frac{2}{\sqrt{5}} \big) \\
& +8 \big( \frac{1}{\sqrt{2}} - \frac{4}{\sqrt{17}} \big)+ 4\big( \frac{-1}{\sqrt{17}} + \frac{1}{\sqrt{2}} \big)  + \frac{9}{\sqrt{17}} + \frac{9}{\sqrt{29}} + \frac{2}{\sqrt{5}} + \frac{8}{\sqrt{17}} +\frac{8}{\sqrt{5}}+\frac{20}{\sqrt{29}} \\
& = 8\big(\pi + \sqrt{2}\big) + 3\big( 2\sqrt{5} - \sqrt{17}\big) + \sqrt{29} \\
& \approx 42.878.
\end{align*}
\end{footnotesize}

\end{example}

\section{3D Algorithm}
\label{sec:3d}

In this Section, we describe an algorithm for computing Expression (\ref{eq:norm}) for 3D-meshes $X$ and $Y$. The algorithm proceeds as follows:
Input: Mesh $X$ with vertices $\{ x_1, \ldots, x_n \}$.
\begin{enumerate}
\item For each vertex $x_i$, compute its star $s_i$.
\item For each star $s_i$, compute the triple intersection points; that is, for all $x_j \in s_i,x_k \in s_i, x_j \neq x_k \neq x_i \neq x_j$, solve $v \cdot x_i = v \cdot x_j = v\cdot x_k$, get a collection of antipodal points on $S^{d-1}$, one pair for each triplet. Label each such point with the triplet $(i,j,k)$ that generated it
\item Triangulate $S^{d-1}$ by placing the vertices with points computed from (2). Connect two vertices if they share two of their labels and the inner products of the points they define is non-negative.
\item For each triangle in the triangulation, evaluate gain in the local ECT in the star. If this is zero, discard the triangle, otherwise record the change.
\item Merge adjacent triangles, that is, triangles that share an edge and that have the same gain. This results in a spherical polygon.
\item Get: For each vertex $x_i$ a collection of polygons $P_k$ with their local gains $a_k$.
\end{enumerate}

This results in having the mesh $X$ in the proto-transform format, i.e.  it is represented as an indicator function of height functions over polygons, which this time, may be overlapping.
After the algorithm, the next step of this algorithm is to use this representation to compute the distances. This resembles by and large the case introduced in the previous section, with relevant details about spherical integration left into the Appendix \ref{secA1}. We will use the rest of section to give mathematical reasoning behind Algorithm 2 as well as discuss some improvement ideas.

The first difference to the brute force algorithm is that we are now computing the polygon intersections inside each link. When mesh resolution is increased, $n$, the number of vertices, increases, but the star of each vertex will have more or less the same number of vertices, for example for primate molars $k \approx 6$. Using the links we need to compute about $n \cdot \binom{k}{3}$ intersections, instead of the $\binom{n}{3}$ intersections that the brute force algorithm requires.

This is based on the following, well-known, observation that allows one to compute changes in Euler Characteristic solely based on local information:
\begin{prop}
Let $X$ be a piecewise linear mesh with vertices $\{x_1,\ldots, x_n\}$ and $v$ a direction such that $h_i:= \langle x_i,v \rangle  \neq \langle x_j,v \rangle $ for $i \neq j$ and $f_v(h)$ the Euler curve in direction $v$. Let also $\epsilon = \min_{i \neq j}|(h_i-h_j)|/2$.
The gain

$$
f_v(h_i+ \epsilon) - f_v(h_i- \epsilon)
$$

associated to vertex $x_i$ in the Euler Curve $f_v$ is completely determined by Star($x_i$) and $v$. 
\end{prop}

\begin{proof}
The Euler characteristic of $X$ is the alternating sum of its simplex counts. Any new simplices that are added to $X$ when $x_i$ is added to it are by definition contained in the star of $x_i$.
\end{proof}

The second difference that the algorithm has is that now the spherical polygons may or may not be disjoint. This is not much of a problem, because regardless of which decomposition we use, we have the following:

\begin{prop}
Suppose shapes $K_1, K_2$ are compactly-supported on unit ball $B^{d}$, let $\mathrm{ECT}(K_1)=\sum_{i=1}^n\alpha_if_i$ and $\mathrm{ECT}(K_2)=\sum_{i=1}^m\beta_jg_j$, we denote $\alpha_{n+j}=-\beta_j$ and $f_{n+j}=g_j$ for simplicity, then we have
\begin{align*}
    d^2(K_1,K_2)&=\int_{S^2}\int_{I}(\mathrm{ECT}(K_1)-\mathrm{ECT}(K_2))^2dtd\sigma\\
    &=\int_{S^2}\int_{I}(\sum_{i=1}^n\alpha_if_i-\sum_{j=1}^m\beta_jg_j)^2 dtd\sigma\\
    &=\int_{S^2}\int_{I}(\sum_{i=1}^{m+n}\alpha_if_i)^2 dtd\sigma\\
    &=\sum_{i,j=1}^{m+n}\alpha_i\alpha_j\int_{S^2}\int_{I}f_if_j dtd\sigma\\
    &=\sum_{i,j=1}^{m+n}\alpha_i\alpha_j\int_{S^2}\int_{I}1_{\{P_i\cap P_j\}}\cdot 1_{[\max \{x(P_i)\cdot v,\, x(P_j) \cdot v\},1]}\ \,dtd\sigma\\
    &=\sum_{i,j=1}^{m+n}\alpha_i\alpha_j\int_{P_i\cap P_j}1-\max \{x(P_i)\cdot v,\, x(P_j) \cdot v\}\ \,dtd\sigma\\
    &=\sum_{i,j=1}^{m+n}\alpha_i\alpha_j\bigg(\mathrm{Vol}(P_i\cap P_j) \\
    & -\int_{P_i\cap P_j\cap \{x(P_i)\cdot v > x(P_j)\cdot v\}}x(P_i)\cdot v\ \,dtd\sigma \\
    & -\int_{P_i\cap P_j\cap \{x(P_i)\cdot v > x(P_j)\cdot v\}}x(P_j)\cdot v\ \,dtd\sigma\bigg).
\end{align*}
\noindent

\end{prop}
To derive the closed formula of ECT distance, it suffices to derive the closed formula of integral of $x\cdot v$ on arbitrary spherical polygons, and a formula for the volume of the intersection of the polygons. In general, it is difficult to provide a closed-form formula for the former spherical integral in $d$ dimensions. The simple case of 2 dimensions was already addressed in the previous section (Proposition \ref{2dheight}), and the three-dimensional formula, which is the most important one for applications, is presented in Appendix \ref{secA1}. Note that in three dimensions the intersection area can be computed with, for example, the Gauss-Bonnet Theorem.

\section{Case Study}
\label{sec:case study}

Now that we have described our algorithm, we demonstrate its utility by applying it to a practical problem. We also compare it to the distances obtained from landmarks as well as a diffeomorphism based method \cite{boyer2011algorithms}. The comparison may be of independent interest to theoreticians pondering the differences between the continuous Procrustes distances and the one obtained from ECT, and to what extent these two are commensurate.

\subsection{The Data}

The dataset comprises 116 primate molars, representing diverse genera and dietary preferences. This dataset was first analyzed in \cite{boyer2011algorithms}. 

\subsubsection{Data Processing}

The meshes are aligned based on landmarks using the Auto3DGM\cite{boyer2015new, puente2013distances} algorithm, with the furthest point sampling and 128 and 256 landmarks for phases 1 and 2, respectively. 

For the purpose of the ECT methods, the meshes were converted to OFF files with meshlab, and further subsampled to 1,000 faces with its quadric edge collapse decimation algorithm. We further center and scale the meshes to the unit ball by a standard procedure.

\subsection{Methods}

We compute the Digital ECT distances based on the algorithm described in this manuscript, using the software provided in Section \ref{code}. The discretized ECT distances are based on 326 directions obtained from a regular subdivision of a regular octahedron. This is a standard procedure to produce close to uniformly spread points on the 2-sphere. For each direction, the discrete Euler Characteristic Transform is evaluated with 100 equally spaced heights between -1 and 1, resulting in a matrix of size $362 \times 100$.

The landmark distance is based on the 2nd phase of the Auto3DGM algorithm that was used to align the meshes. This algorithm makes use of minimum spanning tree to propagate distances, unlike our ECT distances. The continuous procrustes distance is a verbatim copy of the results in \cite{boyer2011algorithms}, reproduced from its accompanying code, provided here for reference.

\subsection{Results}

To see a high-level overview of the different distances and how they cluster the data, multidimensional Scaling Plots of the distances are presented in Figure \ref{MDS plots}. These are 2-D summaries of the distance matrix, projecting points to the plane such that the global distance is distorted as little as possible. From the plots we see that all the methods give, on high level, similar kind of clustering of the dietary groups.

\begin{figure*}
\centering
\begin{subfigure}[b]{0.475\textwidth}
    \centering
    \includegraphics[width=\textwidth]{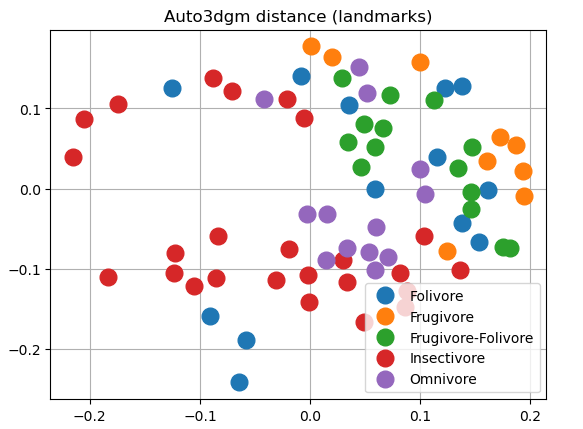}
    \caption[MDS]%
    {{\small Landmarks}}    
    \label{mdsauto3dgm}
\end{subfigure}
\hfill
\begin{subfigure}[b]{0.475\textwidth}  
    \centering 
    \includegraphics[width=\textwidth]{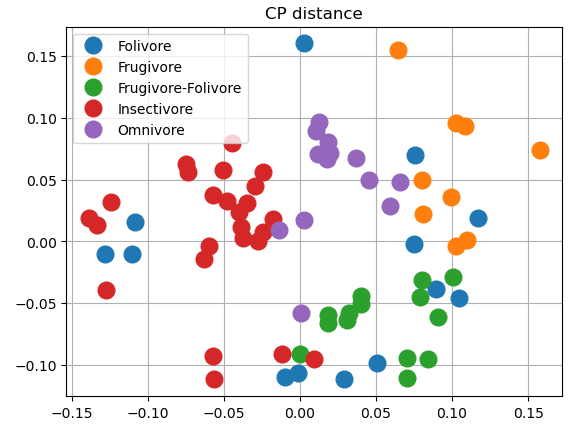}
    \caption[]%
    {{\small Continuous Procrustes Distance}}    
    \label{mdscp}
\end{subfigure}
\vskip\baselineskip
\begin{subfigure}[b]{0.475\textwidth}   
    \centering 
    \includegraphics[width=\textwidth]{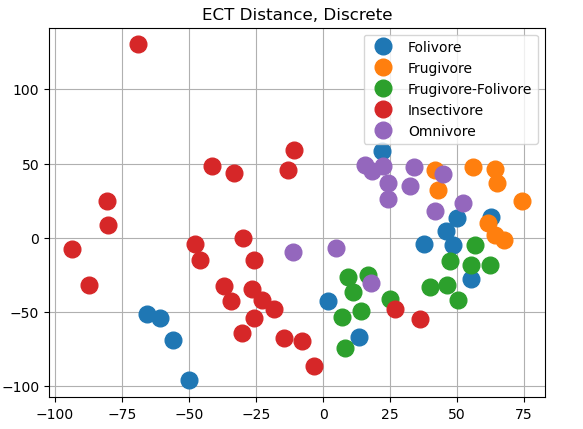}
    \caption[]%
    {{\small Discrete ECT}}    
    \label{mdsdiscrete}
\end{subfigure}
\hfill
\begin{subfigure}[b]{0.475\textwidth}   
    \centering 
    \includegraphics[width=\textwidth]{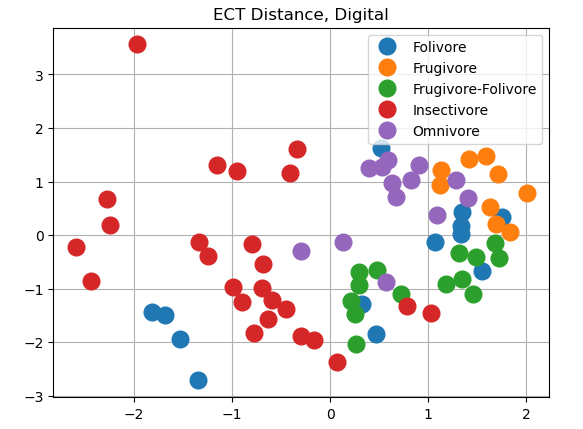}
    \caption[]%
    {{\small Digital ECT}}    
    \label{mdsdigital}
\end{subfigure}
\caption[ Multidimensional Scaling plots  ]
{\small Multidimensional Scaling plots of the different distances. The clustering is quite similar across the methods, with Omnivores and Frugivores forming their own clusters.} 
\label{MDS plots}
\end{figure*}

In Table \ref{table:3} The mantel correlation of the distance matrices are presented. These are standard measures of correlations between distance matrices, accounting for all the pairwise distances. We see that all the four methods are related, as the MDS plots in Figure \ref{MDS plots} hint. Unsurprisingly, we see that the Digital ECT is closely related to the Discrete one, indicating that the discretization that we chose is quite good at capturing the distances as a whole. Interestingly, the digital method has smaller correlation with the two non-ECT methods than the discrete one.

In Figure \ref{tsneplots} we also look at the t-sne plots of the collection computed with different notions of distance. The t-sne, unlike MDS, puts emphasis on small distances, which is a key element in, for example, manifold learning and Gaussian process classifiers, that make use of local structure of the collection in their inference. It is conceivable that there may be relevant small differences that yield to different inference even if the distances are globally in great agreement. While we see some differences between the two methods, the closest neighbors are in much agreement, for example, the frugivores are far from insectivores, and the folivores form three clear subclusters whose relative positions are similar in either display.

\begin{figure*}
\centering
\begin{subfigure}[b]{0.475\textwidth}
    \centering
    \includegraphics[width=\textwidth]{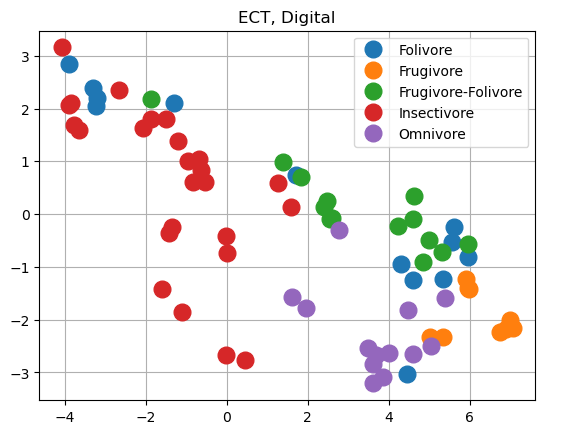}
    \caption[tsne]%
    {{\small Digital  ECT}}    
    \label{tsnedigi}
\end{subfigure}
\hfill
\begin{subfigure}[b]{0.475\textwidth}  
    \centering 
    \includegraphics[width=\textwidth]{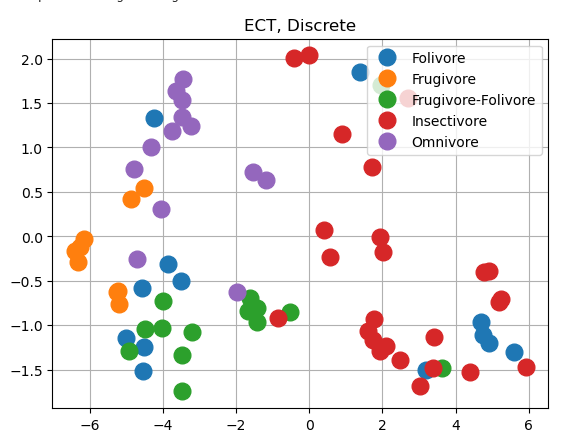}
    \caption[]%
    {{\small Discrete ECT}}    
    \label{tsnedisc}
\end{subfigure}
\caption[ $t$-sne plots  ]
{\small $t$-sne plots based on the distances from the two ECT methods. } 
\label{tsneplots}
\end{figure*}

\begin{table}[!ht]
\caption{Correlation of pairwise distances (the Mantel correlations) of the distance matrices produced by different methods.}
\label{table:3}
\begin{center}
\begin{tabular}{ |c|c|c| } 
 \hline
Method 1 & Method 2 & Correlation \\  \hline
Digital ECT & Discrete ECT & $0.9991$ \\
Digital ECT & Landmarks & $0.6066$\\
Digital ECT & Continuous Procrustes & $0.4896$\\
Discrete ECT & Landmarks & $0.6070$\\
Discrete ECT & Continuous Procrustes & $0.4927$\\
Landmarks & Continuous Procrustes & $0.7373$\\
 \hline
\end{tabular}
\end{center}
\end{table}

In spite of the small differences visible in $t$-sne plots, there is little difference in the clustering behavior, and indeed if we tried to predict, for example, frugivores against other classes using a local smoothing type method, we would get similar inference.

The conclusion of this case study is that the discretization yields results that are very accurate and the inference is for all intents and purposes just as good as with the digital transform. This corroborates the rule of thumb that 326 uniformly spread directions with hundred sublevel sets is typically a good enough discretization for a collection of surface-like 3D-shapes. In general, the procedure outlined in this section can be used to objectively assess how good a given discretization is. One such assessment is presented in Table \ref{table:3b}, where we report the Mantel correlations against the digital algorithm against various discretizations. Ultimately, the discretization must be chosen in terms of whatever problem is studied, what we present here is but a single indication of the discretization error on a high level.

\begin{table}[!ht]
\caption{The Mantel correlations of the distances from different discretization against the true values obtained from the digital algorithm. The discretizations are obtained by subsetting the full discretization of 100 sublevelsets and 326 directions uniformly.}
\label{table:3b}
\begin{center}
\begin{tabular}{ |c|c|c|c|c|c|c| } 
 \hline
Directions $\backslash$ Sublevel sets & 100 & 50 & 25 & 20 & 10 & 5\\ \hline
326 & 0.9991 & 0.9989 & 0.998 & 0.9973 & 0.9923 & 0.9740\\
110 & 0.9979 & 0.9975 & 0.9945 & 0.9923 & 0.9782 & 0.9460\\
38 & 0.9921 & 0.9912 & 0.9838 & 0.9786 & 0.9360 & 0.8481\\
14 & 0.9535 & 0.9487 & 0.9308 & 0.9142 & 0.8173 & 0.7077\\
6 & 0.8715 & 0.8610 & 0.8153 & 0.7921 & 0.6639 & 0.5413\\ \hline
\end{tabular}
\end{center}

\end{table}

\section{Shape Alignment with the Digital Transform}
\label{sec:alignment}

In general, given a collection of shapes $X_1,\ldots,X_n$, the shape alignment problem refers to finding a set of $\binom{n}{2}$ rigid transformations called correspondences $c_{i,j}$ that best align each shape $X_j$ to another shape $X_i$. This is quite a nuanced problem, because the composition of two correspondences $c_{i,j} \circ c_{k,i}$ may not necessarily yield the correspondence between $X_j$ and $X_k$, so some sort of harmonization procedure is in order, for example through a minimum spanning tree. In this section we study the more modest problem of aligning two shapes $X$ and $Y$, which is a first step to studying the more general alignment problem. Solving the more general problem efficiently is left for future.

Concretely, the problem that we want to solve is to find $A \in SO(3)$ that minimizes the distance

\begin{equation}
\label{mindistance}
d_{ECT}(X,AY) = \Big( \int_{-1}^{1} \int_{S^{d-1}} \big( \textrm{ECT}_X(v,h)-\textrm{ECT}_{AY}(v,h) \big)^2 \ dv dh \Big)^{1/2}.
\end{equation}

This is of course the same as maximizing the generalized cross-correlation

$$
\langle X,AY \rangle = \int_{-1}^{1} \int_{S^{d-1}} \textrm{ECT}_X(v,h)\textrm{ECT}_{AY}(v,h) \ dv dh.
$$

One of the immediate advantages that the digital transform has vis-à-vis the discrete one is that it is automatically $SO(d)$-equivariant, and it is almost everywhere differentiable.

If we are using the discrete transform, we have restricted the ECT to a finite set of directions $(v_1, \ldots v_n)$. In particular, we do not have access to the ECT at directions $(Av_1, \ldots Av_n)$, so we are required to approximate the action of $A$ by a permutation of the directions $(v_1, \ldots v_n)$. For example, in \cite{SINATRA} this problem was solved by discretizing $SO(3)$ to a grid of size 16,000, which were approximated by a permutations of the set of 2,918 directions. For each of the 16,000 actions $A_i$, we would then solve the linear assignment problem 

$$
\tilde p_i = \textrm{arg min}_{p:[n] \rightarrow [n]} \sum_{j=1}^{n} (pv_j - A_i v_j)^2,
$$

which is a computationally very expensive operation.

In this Section we propose two algorithms for aligning shapes with \textsc{Ectoplasm}: adaptive grid search and gradient descent. These two methods can be used in tandem. These methods have three advantages over the existing ECT-based alignment schemes. Firstly, it removes the computational overhead of approximating rigid motions with permutations. Secondly, the method involves no permutations, removing an additional source of noise. Thirdly, we are not restricted to any particular a priori chosen grid, so we can simply add more iterations as we go to gain further accuracy without having to redefine the grid and without solving the expensive permutation approximation problem from scratch.

We will set $Y$ be a mesh and $X=AY$ with a randomly chosen $A \in \textrm{SO}(3)$ so that we have a well-defined minimum and we can study the convergence of our algorithms.

\subsection{Adaptive Grid Search}

We begin by discretizing the group of rotations $ \textrm{SO}(3)$. The group of rotations can be parametrized by 3 angles, $(\alpha, \beta, \gamma)$, which can be seen as the quotient space of $\R^3$ with the following identifications:

\begin{align*}
(\alpha,\beta,\gamma) & \simeq (\alpha + 2 \pi,\beta,\gamma) \\
(\alpha,\beta,\gamma) & \simeq (\alpha ,\beta + 2 \pi,\gamma) \\
(\alpha,\beta,\gamma) & \simeq (\alpha +,\beta,\gamma+ 2 \pi) \\
(\alpha,\beta,\gamma) & \simeq (\alpha +\pi,\beta+\pi,\gamma+\pi).
\end{align*}

The angles $\alpha, \beta, \gamma$ are the axis-aligned rotations, and we can associate to each a rotation matrix for example as follows:

$$
A_{\alpha,\beta,\gamma} = \begin{pmatrix}
\cos \alpha & -\sin \alpha & 0\\
\sin \alpha & \cos \alpha & 0\\
0 & 0 & 1
\end{pmatrix} 
\begin{pmatrix}
\cos \beta & 0 & \sin \beta\\
0 & 1 & 0\\
-\sin \beta & 0 & \cos \beta
\end{pmatrix} 
\begin{pmatrix}
1 & 0 & 0\\
0 & \cos \gamma & -\sin \gamma\\
0 & \sin \gamma & \cos \gamma
\end{pmatrix} .
$$

Next we pick an initial grid to start our search. We choose a grid of $8 \times 8 \times 5= 320$ points. Concretely, the grid is the cCrtesian product of the following values for parameters $\alpha, \beta$ and $\gamma$:

\begin{align*}
\alpha & = 0, \pi/4, \pi/2, 3\pi/4, \ldots, 7\pi/4;\\
\beta & = 0, \pi/4, \pi/2, 3\pi/4, \ldots, 7\pi/4;\\
\gamma & =0, \pi/4, \pi/2, 3\pi/4, \pi. 
\end{align*}

After each iteration, we pick the set of angles that produced the maximal cross-correlation. For each element of this maximal angle, we take its two neighbors, and subdivide the resulting interval to 4 pieces. Our new grid is the Cartesian product of these intervals.

This way, each angle is on the same scale, and our grid is of size $4^3=64$, a multiple of 16, leading to efficient parallelization.

As a demonstration of this algorithm, we take as $Y$ one of the shapes analyzed in the previous section, and we take a random element $A$ from $\textrm{SO(3)}$ and set $X=AY$. Using the notation above, in terms of $(\alpha,\beta,\gamma)$, $A$ is represented as approximately $ (4.76,1.07,2.10)$.

We present results of this run in Table \ref{table:adaptivegridsearch}. For each iteration $i$, we report the angles, the ECT inner product, as well as the $\textrm{SO}(3)$ distance between $A$ and its approximation $A_i$.

Recall the $\textrm{SO}(3)$ distance can be computed for two rotation matrices $A$ and $A_i$ as

$$
d(A,A_i) = \arccos \frac{\textrm{trace}(AA^T_{i})-1}{2}.
$$
% seed 10

\begin{table}[!ht]
\caption{Results of the Adaptive Grid Search}
\label{table:adaptivegridsearch}
\begin{tabular}{ |c|c|c|c|c|c| } 
 \hline
 Iteration & $\alpha$ & $\beta$ & $\gamma$ & $\langle X, A_iY \rangle$ & $\textrm{SO}(3)$ Distance \\ \hline
 1 & 4.71239 & 0.78540 & 1.57080 & 22.11458 & 0.56745 \\
 2 & 4.97419 & 1.04720 & 2.35619 & 23.54182 & 0.12827 \\
 3 & 4.79966 & 1.22173 & 2.18166 & 23.20256 & 0.15730 \\
 4 & 4.68330 & 1.10538 & 2.06531 & 24.01233 & 0.05861 \\
 5 & 4.76087 & 1.02781 & 2.14288 & 23.90973 & 0.06097 \\
 6 & 4.81258 & 1.07952 & 2.19459 & 24.14080 & 0.05507 \\
 7 & 4.77811 & 1.04504 & 2.09116 & 24.22715 & 0.03779 \\
 8 & 4.75512 & 1.06803 & 2.11415 & 24.39189 & 0.01881 \\
 9 & 4.73980 & 1.08335 & 2.06818 & 24.37854 & 0.02024 \\
 10& 4.75002 & 1.07313 & 2.09882 & 24.49974 & 0.00863 \\
 11& 4.75683 & 1.06632 & 2.09201 & 24.51984 & 0.00783 \\
  \hline
 Truth & 4.75875 & 1.07199 & 2.09902 & 24.60740 & 0.00000 \\
  \hline
\end{tabular}
\end{table}

% seed 1
%\begin{table}[!h]
%\caption{Results of the Adaptive Grid Search}
%\label{table:adaptivegridsearch2}
%\begin{tabular}{ |c|c|c|c|c|c| } 
% \hline
% Iteration & $\alpha$ & $\beta$ & $\gamma$ & $\langle X, A_iY \rangle$ & $\textrm{SO}(3)$ Distance \\ \hline
% 1 & 3.92699 & 0.7854 & 0.0 & 22.37717 & 0.50063 \\
% 2 & 4.18879 & 0.5236 & -0.2618 & 22.75489 & 0.22715 \\
% 3 & 4.36332 & 0.69813 & -0.08727 & 23.02164 & 0.14906 \\
% 4 & 4.24697 & 0.81449 & -0.20362 & 23.78182 & 0.07808 \\
% 5 & 4.1694 & 0.73692 & -0.28119 & 24.23828 & 0.02972 \\
% 6 & 4.22111 & 0.78863 & -0.22948 & 24.08293 & 0.04729 \\
%7 & 4.18664 & 0.75415 & -0.26395 & 24.40722 & 0.01908 \\
%8 & 4.20962 & 0.73117 & -0.28694 & 24.29229 & 0.03156 \\
%9 & 4.1943 & 0.74649 & -0.27162 & 24.49241 & 0.00795 \\
%10 & 4.20451 & 0.75671 & -0.2614 & 24.51071 & 0.0078 \\
%11 & 4.1977 & 0.7499 & -0.26821 & 24.54139 & 0.00494 \\
%12 & 4.20224 & 0.75444 & -0.26367 & 24.54052 & 0.00564 \\
%13 & 4.20527 & 0.75141 & -0.2667 & 24.57984 & 0.00244 \\
%14 & 4.20325 & 0.74939 & -0.26468 & 24.58779 & 0.00182 \\
%15 & 4.2046 & 0.75074 & -0.26602 & 24.5948 & 0.00106 \\
%16 & 4.2037 & 0.74984 & -0.26513 & 24.59765 & 0.00091 \\
%17 & 4.2043 & 0.75044 & -0.26573 & 24.60084 & 0.00049 \\
%  \hline
% Truth & 4.2039441 & 0.75004812 & -0.26582519 & 24.60740 & 0.00000 \\
%  \hline
%\end{tabular}
%\end{table}

From Table \ref{table:adaptivegridsearch} we see that the inner product is inversely proportional to the $\textrm{SO}(3)$ Distance, so that the inner product is in practice a useful metric for the shape alignment problem, and the algorithm does not, for example, seek a local minimum. The method seems to be able to find the correct alignments using relatively fast.

\subsection{Gradient Descent}

The digital version of the shape alignment problem is, in theory, almost differentiable, and the piecewise components of it are easily described with even pen and paper. However, writing the gradient down exactly is quite a complicated task. To this end, we also implement an auto-differentiable version of the code, which allows us to compute the gradient by abusing the simplicity of individual operators in conjunction with the chain rule. In order to do this, we provide a proof-of-concept implementation of the algorithm with torch functions, a popular machine-learning framework that has built-in automatic differentials for basic operations. Due to the cost of the computational resources required for an efficient implementation, we demonstrate a proof-of-concept of this procedure by applying the algorithm on a simple non-convex polyhedron on 9 vertices and 14 faces obtained from simplifying a shape studied in the previous section. This is a mesh with high curvature.

The ECT representation introduced in this paper and implemented in the code package \textsc{Ectoplasm} allows for straightforward implementation of both vanilla gradient descent and also stochastic gradient descent.

In terms of this representation, write $ECT_X(v,h) = \sum_{i=1}^{n} \alpha_i P(x_i,v_i)$ and \\ $ECT_Y(v,h) = \sum_{j=1}^{n} \alpha_j P(y_j,v_j)$, the alignment problem corresponds to finding the optimal $A$ that maximizes

\begin{align*}
\langle X,AY \rangle &  = \int_{-1}^{1} \int_{S^{d-1}} \textrm{ECT}_X(v,h)\textrm{ECT}_{AY}(v,h) \ dv dh \\
& =  \int_{-1}^{1} \int_{S^{d-1}} \sum_{i=1}^{n} \alpha_i P(x_i,v_i) \sum_{j=1}^{m} \alpha_jP(Ay_j,Av_j) \ dv dh.
\end{align*}

The gradient of this expression with respect to the vector angles $(\alpha, \beta, \gamma)$ is then readily available through auto-differentiation, and we can seek the optimum via gradient ascent.

For the stochastic gradient, we have an additional choice to make. We could divide up the double sum over one index or the other, or even both of them. The latter is not necessarily a good option\footnote{At least if sampling uniformly at random.}, because then we would at each iteration compute the gradient based on the intersection of two sparse sets of spherical polygons on $S^2$. The gradient would then point out to whichever direction maximizes the overlap between these two sets. This type of sampling would lead to large volatility that would override the signal. By subsetting only one index, we ensure the other set of polygons is full, i.e., there is a polygon (possibly several) in each direction, so that the intersection cannot be maximized solely by seeking maximal overlap.

We seek the alignment by vanilla gradient ascent, which we initialize with angles $\theta_0:=(\alpha,\beta,\gamma)=(0,0,0)$. We update the angles $\theta_i$ of alignment $A_i$ with the gradient ascent

$$
\theta_{i+1} = \theta_{i} + \lambda \frac{\nabla \theta_i}{ ||\nabla \theta_i ||},
$$
with $\lambda=1$ for the first 30 iterations and $0.1$ after the 20 following iterations and 0.01 thereafter. The results are presented in Table \ref{table:gradientascent}.
\begin{table}[!ht]
\caption{Results of the Gradient Ascent}
\label{table:gradientascent}
\begin{tabular}{ |c|c|c|c|c|c| } 
 \hline
 Iteration & $\alpha$ & $\beta$ & $\gamma$ & $\langle X, A_iY \rangle$ & $\textrm{SO}(3)$ Distance \\ \hline
0 & 0 & 0 & 0 & 29.4329 & 2.90493\\ 
%1 & 0.422 & -0.4179 & 0.8045 & 29.76868 & 2.91164\\ 
%2 & 0.0248 & 0.4991 & 0.8421 & 29.79484 & 2.70202\\ 
%3 & 0.8587 & -0.0526 & 0.8577 & 29.85138 & 2.77667\\ 
%4 & 0.3512 & 0.8071 & 0.9147 & 29.88236 & 2.79659\\ 
5 & 1.2458 & 0.3611 & 0.8870 & 29.8906 & 2.59675\\ 
%6 & 0.4676 & 0.9384 & 1.1342 & 29.89819 & 2.65968\\ 
%7 & 1.4215 & 0.672 & 0.996 & 29.91188 & 2.60158\\ 
%8 & 0.5539 & 0.965 & 0.5941 & 29.83583 & 3.07285\\ 
%9 & 0.6352 & 0.2831 & 1.321 & 29.9303 & 2.90833\\ 
10 & 0.6660 & 0.8395 & 0.4906 & 29.8233 & 2.84400\\ 
%11 & 0.5998 & 0.6269 & 1.4655 & 29.89882 & 2.61581\\ 
%12 & 1.1929 & 0.5261 & 0.6667 & 29.90655 & 2.47108\\ 
%13 & 1.4875 & 1.3959 & 1.0625 & 29.4366 & 2.75089\\ 
%14 & 1.7397 & 0.4524 & 0.8474 & 29.82167 & 2.19111\\ 
15 & 1.2837 & 1.0522 & 0.1899 & 29.7568 & 2.02619\\ 
%16 & 1.7758 & 0.2345 & -0.1088 & 30.05862 & 1.38949\\ 
%17 & 1.8819 & 1.1995 & -0.3489 & 30.43355 & 0.96513\\ 
%18 & 2.6094 & 0.9819 & -0.9996 & 31.21437 & 0.38933\\ 
%19 & 1.8653 & 0.9406 & -0.3328 & 30.43691 & 0.96245\\ 
20 & 2.6830 & 0.9540 & -0.9083 & 31.1978 & 0.37889\\ 
%21 & 1.8611 & 1.1142 & -0.3617 & 30.44132 & 0.95473\\ 
%22 & 2.6316 & 0.9932 & -0.9875 & 31.1888 & 0.39917\\ 
%23 & 1.8709 & 0.9431 & -0.3403 & 30.44594 & 0.95002\\ 
%24 & 2.6823 & 0.9696 & -0.9243 & 31.1814 & 0.39113\\ 
25 & 1.8636 & 1.0757 & -0.3600 & 30.4473 & 0.94800 \\ 
30 & 2.6630 & 0.9739 & -0.9524 & 31.1811 & 0.39757 \\ \hline
35 & 2.3444 & 1.0060 & -0.7960 & 32.0700 & 0.06131 \\
40 & 2.4250 & 1.0194 & -0.8537 & 32.0291 & 0.07331 \\
45 & 2.3589 & 0.9934 & -0.7833 & 32.0694 & 0.06037 \\
50 & 2.4251 & 1.0178 & -0.8535 & 32.0321 & 0.07292 \\ \hline
55 & 2.3914 & 1.0061 & -0.8184 & 32.2554 & 0.00587 \\
60 & 2.3894 & 1.0097 & -0.8083 & 32.2477 & 0.00717 \\
65 & 2.3911 & 0.9995 & -0.8162 & 32.2475 & 0.00726 \\
70 & 2.3861 & 1.0089 & -0.8149 & 32.2600 & 0.00398 \\
  \hline
 Truth & 2.3886 & 1.0059 & -0.8151 & 32.2767 & 0.00000 \\
 %Truth & 2.38862675 & 1.00593583 & -0.81512331 & 32.2767 & 0.00000 \\
  \hline
\end{tabular}
\end{table}

While the gradient ascent seems to find the correct alignment without very complicated optimization of $\lambda$, it does suffer from the fact that there are many local minima, which in part due to the high curvature of the mesh. From the oscillation of the loss in Table \ref{table:gradientascent} there are signs that $\lambda$ is set too high, but this is somewhat necessary so as to allow enough movement past local minima. The gradient ascent therefore seems much more useful in determining the fine alignment, after a cruder alignment from, for example, the adaptive grid search.

\section{Code availability}
\label{code}
The \textsc{Ectoplasm} code is available on: \\
\url{https://github.com/hkirvesl/ECTOPLASM}.

\section{Discussion and Future directions}

In this work we have introduced a digitalization of the Euler Characteristic Transform and developed a functioning algorithm for computing distances with it in a fully digital manner. This is just a first step to making full use of the rich mathematical structure in the transform. An interesting an important question is how one might use this transform for more complicated statistical problems, such as subshape analysis, in the spirit of \cite{SINATRA}.

One of the key enablers of the digitalization algorithm is the inclusion exclusion property of the Euler Characteristic. In theory, it would be quite straightforward to apply the same procedure for the Persistent Homology Transform. However, computing it in practice 
is a much harder problem due to the complexity of computations. But this would be very interesting due to the connections to the sheafification of shape space as outlined in \cite{arya2024sheaf}.

The differentiability of the alignment problem provides another connection to studying the shape space. In this work, we only considered transformations in the group $\textrm{SO}(3)$ for the specific problem of aligning shapes. But maybe we could find another more flexible group of transforms that might give us a way to find paths between shapes via gradients.  A simple minded idea would be to differentiate the penalty with respect to the vertex locations, but care needs to be taken so as not to allow for self-intersections. 

Another problem of applied interest would be extending the algorithm to soft tissue, as in \cite{kirveslahti2024representing}. This is again theoretically conceivable, but hard in practice.

\section*{Acknowledgements}

The authors would like to thank Nicolas Berkouk, Sayan Mukherjee, Justin Curry, Shreya Arya, Kathryn Hess, Robert Ravier, Wojciech Chach\'olski and Heather Harrington for helpful discussions and comments.

HK developed a 2-dimensional prototype of the algorithm during his visit to the Quantum Center at University of Southern Denmark in 2022. The author would like to thank the Center for their hospitality. The bulk of the work was carried out at EPFL.

\bibliographystyle{plain}
\bibliography{refs.bib}

\appendix

\section{Spherical integration formulae}
\label{secA1}

In this Appendix we detail how we integrate the height functions over spherical polygons in three dimensions.

We begin by fixing spherical coordinates for points on the unit sphere $S^2$, i.e. a homeomorphism $$F:(0,2\pi)\times(-\pi,\pi)\rightarrow S^2-I,\ (\phi, \tau)\mapsto (\cos{\tau}\cos{\phi},\ \cos{\tau}\sin{\phi},\ \sin{\tau}),$$ where $I$ is the union of Prime Meridian, North Pole and South Pole. In geometry, $\phi$ is the azimuthal angle and $\tau$ is the signed angle between the vector from the origin to that point and the $xy$-plane. We now provide a parametrization of great circles under this coordinate system.
\begin{lemma}[\cite{247336}]
\label{lem:parameterization}
    Under the chosen spherical coordinate, a non-polar great circle on a unit sphere can be characterized by \[\tan\tau=a \cos(\phi-\phi_0),\]in which we denote $a=\tan\tau_0$ and $(\phi_0,\tau_0)$ refers to the point with the minimum z-coordinate on the great circle.
\end{lemma}
\begin{proof}
    Firstly note that $\tan\tau=0$ characterizes the equator, where $a=tan\tau_0=0$. If the great circle, denoted by $G$, is not the equator, there exists only one point with the minimum z-coordinate on it. Applying a counterclockwise rotation by $\phi_0$ to $G$, denoted by $\rho_{\phi_0}G$, we notice that the unit vector ${\bf u}=[\sin(-\tau_0), 0,\cos(-\tau_0)]$ is normal to the plane that $\rho_{\phi_0}G$ lies on. Therefore \[{\bf u}\cdot [\cos{\tau}\cos{(\phi-\phi_0)},\  \cos{\tau}\sin{(\phi-\phi_0)},\ \sin{\tau}]=0,\]and hence we get $\tan\tau = a\cos(\phi-\phi_0)$.
\end{proof}
\noindent
With this parametrization for great circles, we can compute the following spherical integration we need for spherical polygons.
\begin{prop}
    Given a spherical polygon $SP_i=\{(\phi_1, \tau_1), (\phi_2,\tau_2),...,(\phi_n,\tau_n)\}$ and associated vertex $p_i=(\phi_i, \tau_i)$, the ordering of vertices of $SP_i$ is compatible with the orientation of the polygon, namely $[\frac{\partial}{\partial \phi},\frac{\partial}{\partial \tau}]$. If $(SP_i,p_i)$ satisfies 
    \begin{itemize}
        \item[1.] $p_i\notin I$;
        \item[2.] $SP_i$ does not intersect with $I$, equator and any meridian.
    \end{itemize}
    we denote $\phi_{n+1}=\phi_1$ and then have
    \[I(p_i)=\int_{SP_i} {\bf op_i}\cdot {\bf v}\ d\phi d\tau=\sum_{k=1}^n \frac{1}{4}\sin\tau_i\cdot I_{k1}-\frac{1}{4}\cos\tau_i\cdot I_{k2}-\frac{1}{2}\cos\tau_i\cdot I_{k3},\]
    where we set the origin to be $o$, ${\bf v}$ is the direction vector and
    \begin{align*}
  I_{k1} &= \int^{\phi_{k+1}}_{\phi_k}\frac{1-a^2\cos^2(\phi-\phi_0)}{1-a^2\cos^2(\phi-\phi_0)}\ d\phi,\\
  I_{k2} &= \int^{\phi_{k+1}}_{\phi_k}\frac{2a\cdot\cos(\phi-\phi_0)\cos(\phi-\phi_i)}{1+a^2\cos^2(\phi-\phi_0)}\ d\phi,
  \\
  I_{k3} &= \int^{\phi_{k+1}}_{\phi_k}\arctan(a\cdot\cos(\phi-\phi_0))\cos(\phi-\phi_i)\ d\phi.
\end{align*}
\end{prop}
\begin{proof}
    According to \cite{kells1940}, the angle $\sigma$ between direction vector ${\bf v}=(\phi,\tau)$ and $\bf op_i$ satisfies \[\cos\sigma=\sin\tau_i\sin\tau+\cos\tau_i\cos\tau\cos(\phi-\phi_i).\]
    Notice that ${\bf op_i}\cdot{\bf v}=\cos\sigma$, we have
    \begin{align*}
        \int_{SP_i} {\bf op_i}\cdot {\bf v}\ d\phi d\tau=\int_{SP_i} \sin\tau_i\sin\tau+\cos\tau_i\cos\tau\cos(\phi-\phi_i)\ d\phi d\tau.
    \end{align*}
    Denote $\omega = \frac{1}{4}\sin\tau_i\cdot\cos 2 \tau -\frac{1}{4}\cos\tau_i\cdot\sin 2\tau\cdot\cos(\phi-\phi_i)-\frac{1}{2}\cos\tau_i\cdot\cos(\phi-\phi_i)$ the 1-form on $SP_i$, we observe that $d\omega$ is exactly the 2-form being integrated on the right-hand side. By Stokes Theorem, we have
    \[\int_{SP_i} {\bf op_i}\cdot {\bf v}\ d\phi d\tau=\int_{SP_i}d\omega=\int_{\partial (SP_i)}\omega=\sum_{k=1}^n\int_{e_k}\omega,\]where $e_k$ is the edge of $SP_i$ from $(\phi_k, \tau_k)$ to $(\phi_{k+1}, \tau_{k+1})$. We then use the tangent half-angle formula to replace $\cos 2\tau$, $\sin2\tau$ by $\tan\tau$ and plug the parametrization in lemma \ref{lem:parameterization} into $\omega$.
    \begin{footnotesize}
    \begin{align*}
        \int_{e_k} \omega &= \int_{e_k} \Big(\frac{1}{4}\sin{\tau_i}\cdot\frac{1-\tan^2\tau}{1+\tan^2\tau}-\frac{1}{4}\cos{\tau_i}\frac{2\tan\tau}{1+\tan^2\tau}\cos{(\phi-\phi_i)}-\frac{1}{2}\cos{\tau_i}\cdot\tau\cdot\cos{(\phi-\phi_i)}\Big)\ d\phi\\
        &= \frac{1}{4}\sin\tau_i\cdot I_{k1}-\frac{1}{4}\cos\tau_i\cdot I_{k2}-\frac{1}{2}\cos\tau_i\cdot I_{k3}.
    \end{align*}
    \end{footnotesize}
\end{proof}
\noindent
As we can see, there are restrictions on the choice of $(SP_i,p_i)$ to make the coordinate system and parametrization applicable. In practice, it is always possible and efficient to apply a rotation to $(SP_i, p_i)$ to satisfy the restrictions. Hence the formula above is enough for us to compute the desired integration on any kind of spherical polygon. 

\end{document}